\documentclass{amsart}
\usepackage{graphicx}
\usepackage{latexsym}
\usepackage{color}
\usepackage{amsfonts}
\usepackage[all]{xy}
\usepackage{amssymb, mathrsfs, amsfonts, amsmath}
\usepackage{amsbsy}
\usepackage{amsfonts}
\newtheorem{corollary}{Corollary}

\newtheorem{lemma}{Lemma}
\newtheorem{proposition}{Proposition}
\newtheorem{remark}{Remark}
\newtheorem{theorem}{Theorem}

\numberwithin{equation}{section}

\newcommand{\be}{\begin{equation}}
	\newcommand{\ee}{\end{equation}}
\newcommand{\ben}{\begin{enumerate}}
	\newcommand{\een}{\end{enumerate}}
\newcommand{\beq}{\begin{eqnarray}}
	\newcommand{\eeq}{\end{eqnarray}}
\newcommand{\beqn}{\begin{eqnarray*}}
	\newcommand{\eeqn}{\end{eqnarray*}}

\usepackage[pagebackref]{hyperref}

\begin{document}
	\title{On a generalized Finsler warped product metrics with vanishing Landsberg curvature}

 \author{Newton Sol\'orzano}
 \address[N. Sol\'orzano]{ILACVN - CICN, Universidade Federal da Integra\c c\~ao Latino-Americana, UNILA}
 \curraddr{Itaipu Parquetec, Foz do Igua\c cu-PR, 85867-970 - Brasil}	\email{nmayer159@gmail.com}

 \author{Dik D. Lujerio Garcia}
 \address[D. Lujerio]{Departamento Académico de Matemática de la Facultad de Ciencias, Universidad Nacional Santiago Ant\'unez de Mayolo, UNASAM}
 \curraddr{Jr. Augusto B.Leguia N° 110, Huaraz - Per\'u}	\email{dlujeriog@unasam.edu.pe}	
 
  \author{V\'ictor Le\'on}
 \address[V. Le\'on]{ILACVN - CICN, Universidade Federal da Integra\c c\~ao Latino-Americana, UNILA}
 \curraddr{Parque tecnol\'ogico de Itaipu, Foz do Igua\c cu-PR, 85867-970 - Brasil}
 \email{vicaml19@gmail.com}

 \author{Alexis Rodr\'iguez Carranza}	
 \address[A. Rodr\'iguez Carranza]{Departamento de Ciencias, Universidad Privada del Norte, UPN}
 \curraddr{Sede San Isidro, Av. El Ejercito 920, Trujillo - Per\'u}	\email{alexis.rodriguez@upn.edu.pe}	
 
	\begin{abstract}
    In this paper, we study weakly orthogonally invariant Finsler metrics and derive explicit expressions for their Berwald and Landsberg curvatures. We then obtain the system of partial differential equations characterizing generalized Finsler warped product metrics on $I \times \mathbb{R}^n$, which form a subclass of weakly orthogonally invariant Finsler metrics, under the conditions of vanishing Berwald and Landsberg curvatures. As an application, we construct examples of non-regular Landsberg metrics that are not of Berwald type.
 
	\end{abstract}
	
	\date{\today}

	\maketitle
\section{Introduction}

In the classification of Finsler manifolds, the relationship between Berwald metrics and Landsberg metrics remains one of the most significant open questions in the field. By definition, a Finsler metric is Berwaldian if its Berwald curvature vanishes ($B^i_{jkl} = 0$), and it is Landsbergian if its Landsberg curvature vanishes ($L_{jkl} = 0$). While it is well-established that every Berwald metric is a Landsberg metric, the converse—whether every Landsberg metric is necessarily Berwaldian—has remained an unsolved problem for decades, famously dubbed the search for ``unicorns'' by D. Bao \cite{DBao2007}. 

If one considers Finsler metrics without the regularity condition, the statement is no longer true. In this setting, several interesting examples arise. The first non-regular unicorns were constructed by Asanov \cite{Asanov2006}, and later generalized by Shen \cite{Shen2009}. More recently, additional explicit examples have been provided by Elgendi \cite{Elgendi2021,Elgendi2025,Elgendi2026}. 

Beyond their intrinsic geometric relevance, non-Berwaldian Landsberg metrics have gained increasing attention due to their role in Finsler gravity. Indeed, Finslerian extensions of Einstein's theory have been developed in \cite{PfeiferWohlfarth2012,HohmannPfeiferVoicu2019}, where field equations for Finsler spacetimes are formulated. Within this framework, non-regular Landsberg metrics (unicorns) are of particular importance, since they provide exact solutions of these field equations.

A major recent development in this direction is the construction of explicit cosmological solutions based on unicorn metrics \cite{Heefer2023}. In this work, the authors exhibit a non-Berwaldian Landsberg spacetime with homogeneous and isotropic symmetry, which solves the Finsler gravity field equations and describes a dynamically expanding universe. This result is particularly significant, as it shows that nontrivial cosmological dynamics in Finsler gravity cannot arise in the Berwald setting. Indeed, it has been observed that Berwald-type cosmological solutions are necessarily static, implying that any physically meaningful dynamical cosmological model within the Landsberg class must be non-Berwaldian. 

Moreover, these constructions reveal that physically viable causal structures, such as well-defined light cones, may exist even when the fundamental tensor is not of Lorentzian signature. This highlights the richness of the Finslerian framework and further emphasizes the importance of non-regular unicorn metrics, not only for addressing a long-standing geometric problem, but also for constructing non-trivial and physically meaningful models in Finsler gravity \cite{Heefer2023,HohmannPfeiferVoicu2019,PfeiferWohlfarth2012}.

On the other hand, there exist important Finsler metrics in the literature which satisfy 
\begin{align}\label{eq:1}
F\left((x^0,O\overline{x}),(y^0,O\overline{y})\right)=F\left((x^0,\overline{x}),(y^0,\overline{y})\right), \quad \text{for every } O\in O(n),
\end{align}
where 
$ x=(x^0,\overline{x})\in M=I\times \mathbb{R}^n$ and $y=(y^0,\overline{y})\in T_xM$. Examples include the Shen’s fish tank metric, spherically symmetric metrics, and warped product metrics \cite{HM1,Z,Zhao2018,Kozma2001,Liu2019,marcal2023}. 

A Finsler metric $F$ is called weakly orthogonally invariant \cite{Liu2024} (or cylindrically symmetric  in an alternative terminology in \cite{Solorzano2023}) if it satisfies \eqref{eq:1}. In \cite{Liu2024}, it was shown that such metrics are non-trivial, in the sense that they are not necessarily orthogonally invariant.

In \cite{Solorzano2022}, it was proved that every cylindrically symmetric Finsler metric can be written as 
\[
F(x,y)=\vert \overline{y}\vert \phi\left(x^0,\vert \overline{x}\vert,\frac{\langle\overline{x},\overline{y}\rangle}{\vert\overline{y}\vert},\frac{y^0}{\vert \overline{y}\vert}\right),
\]
and in \cite{Solorzano2023}, necessary and sufficient conditions for this expression to define a Finsler metric were obtained.

Regarding warped product metrics, in \cite{zheng2024} differential equations characterizing Landsberg metrics were derived, leading to new non-regular unicorn examples. In \cite{Elgendi2026}, further explicit families of unicorns were constructed within spherically symmetric metrics. In \cite{Solorzano2022}, it was shown that no regular unicorns exist in a certain warped setting, while in \cite{LiuMo2025}, a system characterizing weakly orthogonal invariant metrics with vanishing Berwald curvature was obtained.

Motivated by these developments, in this paper we compute the Berwald and Landsberg curvatures using a convenient notation ($\Psi$), which significantly simplifies the calculations. By restricting to a subclass of weakly orthogonally invariant Finsler metrics, we prove that there are no regular unicorns within this class.

 \begin{theorem}\label{theo:main}\normalfont
     Let $F=\vert\overline{y}\vert{\phi(x^0,r,z)}$ be a Finsler metric defined on  $M=I\times \mathbb{B}^n(\rho)\subset \mathbb{R}\times\mathbb{R}^n,$ where $ z=\frac{y^0}{\vert\overline{y}\vert}, $ $r=\vert\overline{x}\vert$  and  $TM $  with coordinates   \eqref{coordx}-\eqref{coordy}.  Then $F$ is Landsberg metric if and only if it is Berwald metric.
 \end{theorem}

 In contrast of this theorem, we obtain  non-regular unicorn Finsler metrics: 
 
 Let $k=k(x^0)>0,$ $g_1=g_1(x^0,r),\, g_2=g_2(x^0,r)\neq 0 $ and $g_3=g_3(x^0,r)$ be differentiable functions such that 
 \begin{align}\label{conditions01}
 \Delta = g_1g_3-g_2^2&> 0, & \left[\frac{g_2}{\sqrt{\Delta}}\right]_r&=0, & \left[\frac{g_2^2}{g_1g_3}\right]_r&=0.
 \end{align}
 Then, the Finsler metric $F$ given by
 \begin{align}\label{eq:phinlandsintro}
\displaystyle F(x,y)
&=
k\;
\sqrt{\frac{1}{g_1}\left[g_1(y^0)^2 + 2g_2y^0\vert \overline{y}\vert + 2g_3\vert \overline{y}\vert^2\right]}\;
\times e^{
\frac{g_2}{\sqrt{\Delta}}\,\arctan\!\left(\frac{g_1y^0 + g_2\vert \overline{y}\vert}{\sqrt{\Delta}\vert\overline{y}\vert}
\right)},
\end{align}
is Landsberg metric. Furthermore,  $F$ given by \eqref{eq:phinlandsintro} with the conditions \eqref{conditions01} is Berwald metric if, and only if, 
\begin{align*}
    \left[\frac{g_2}{\sqrt{\Delta}}\right]_{x^0}&=0,& \left[\frac{g_3}{g_2}\right]_{x^0} + 2\frac{k'}{k}\left[\frac{g_3}{g_2}\right]&=0.
\end{align*}

 In Section 2 we give some preliminaries and recall some recent results about weakly orthogonally invariant Finsler metrics. In section 3 we study their Berwald and Landsberg curvatures. Specifically, we obtain their the Berwald and Landsberg curvatures (see Theorem \ref{Bcurv} and Theorem \ref{Lcurv}) in the same section, we give  the PDE system that characterizes  weakly orthogonally invariant Finsler metrics with vanishing Berwald curvature (see Theorem \ref{Thm2:B0}), which is a refinement of Proposition 3.1 in \cite{LiuMo2025}. In section 4 we  prove that if $F=u\phi(x^0,r,z)$ is a non-Randers-Berwald metric, then  $\phi$ must satisfy \eqref{eq:Berwald01} - \eqref{eq:Berwald02}.

Also in Section 4 we prove Theorem \ref{theo:main} and in Section 5 we obtain \eqref{eq:phinlandsintro}.

\section{Preliminaries}

In this section, we give some notations, definitions, and lemmas that will be used in the proof of our main results.
Let $M$ be a manifold, and let $TM=\cup_{x\in M}T_xM$ be the tangent
bundle of $M$, where $T_xM$ is the tangent space at $x\in M$. We
set $TM_o:=TM\setminus\{0\}$ where $\{0\}$ stands for
$\left\{(x,\,0)|\, x\in M,\, 0\in T_xM\right\}$. A {\em Finsler
metric} on $M$ is a function $F:TM\to [0,\,\infty)$ with the
following properties:
\begin{itemize}
    \item[(a)] $F$ is $C^{\infty}$ on $TM_o$;

\item[(b)] At each point $x\in M$, the restriction $F_x:=F|_{T_xM}$ is a
Minkowski norm on $T_xM$.
 
\end{itemize}
Let  $\mathbb{B}^n(\rho)\subset\mathbb{R}^n$ the $n$ dimensional open ball of radius $\rho$ and centered at the origin ($n\geq 2$). Set 
 $M=I\times \mathbb{B}^n(\rho)\subset \mathbb{R}\times\mathbb{R}^n,$ with coordinates on $ TM $
\begin{align}
	x&=(x^0, \overline{x}), \quad \overline{x}=(x^1,\ldots,x^n),\label{coordx}\\
	y&=(y^0, \overline{y}), \quad \overline{y}=(y^1,\ldots,y^n).\label{coordy}
\end{align} 
Throughout our work, the following convention for indices is adopted: 
\begin{align*}
	0\leq&A, B, \ldots \leq n;\\
	1\leq&i,j,\ldots \leq n.
\end{align*} 
Introducing the notation 
\begin{align}\label{rs} 
	 r&:=|\overline{x}|,  &s&:=\frac{\langle \overline{x},\,\overline{y}\rangle}{\vert\overline{y}\vert},&  z&:= \frac{y^0}{\vert\overline{y}\vert},
\end{align} 
where $|\cdot|$ and $\langle\cdot,\cdot\rangle$ are, respectively, the 
standard Euclidean norm and inner product on $\mathbb{R}^n$. 

In \cite{Solorzano2022}, the authors proved that, if the Finsler metric $F$ satisfies \eqref{eq:1}, then there exist a positive function $\phi:\mathbb{R}^4 \to\mathbb{R} $ such that,
\begin{align}\label{def:F}
F(x,y) = \vert \overline{y}\vert \phi(x^0,r,s,z).
\end{align}
On the other hand, defining $ \Omega $ and $ \Lambda $ as,
\begin{align}
	\Omega:=&\phi-s\phi_s-z\phi_z \label{DefOmega},\\
	 \Lambda:=& \Omega \phi_{zz}+(r^2-s^2)(\phi_{ss}\phi_{zz}-\phi^2_{sz}),\label{Def:Lambda}
\end{align}
where, the sub-index $s,z$ are the partial derivatives respect to $s$ and $z$ respectively, in \cite{Solorzano2023} was proved the next result about the necessary and sufficiency condition for the function $F=\vert \overline{y}\vert \phi(x^0,r,s,z)$ to be a Finsler metric.
\begin{proposition}\normalfont
		Let $F=\vert\overline{y}\vert{\phi(x^0,r,s,z)}$ be a Finsler metric defined on $ M $, where $ z=\frac{y^0}{\vert\overline{y}\vert}, $ $r=\vert\overline{x}\vert$, $ s=\frac{\langle\overline{x},\overline{y}\rangle}{\vert\overline{y}\vert} $ and  $TM $  with coordinates   \eqref{coordx}-\eqref{coordy}. Then $  F$ is a Finsler metric if, and only if, the positive function $ \phi $ satisfies  $ \Lambda>0 $ for $ n=2 $  with additional inequality, $ \Omega>0 $ for $ n\geq 3. $
\end{proposition}
The next proposition  gives us one the most important quantities in Finsler Geometry: The geodesic coefficients \begin{align*}
	 G^A=Py^A+Q^A, \end{align*}
where
\begin{align*}
	P:=&\frac{F_{x^C}y^C}{2F}, & Q^A:=\frac{F}{2}g^{AB}\left\{F_{x^Cy^B}y^C-F_{x^B}\right\},
\end{align*}
and $g^{AB}$ is the inverse  of the matrix $g_{AB}$ (see details in  \cite{Solorzano2023}).

\begin{proposition}\label{prop2}\normalfont Let
	$F=\vert\overline{y}\vert{\phi(x^0,r,s,z)}$ be a Finsler metric defined on $ M $, where $ z=\frac{y^0}{\vert\overline{y}\vert}, $ $r=\vert\overline{x}\vert$, $ s=\frac{\langle\overline{x},\overline{y}\rangle}{\vert\overline{y}\vert} $ and  $TM $  with coordinates   \eqref{coordx}-\eqref{coordy}.  Then the geodesic spray coefficients $ G^A $ are given by
	\begin{align}
        	G^0&=u^2N,\label{eq:G^0}\\
		G^i&=u^2Wu_i+ u^2Ux^i,\label{eq:G^i}
	\end{align}
where $ u=\vert\overline{y} \vert$, $u_i=\frac{\partial u}{\partial y^i}=\frac{y^i}{u}$, $ \Omega, \Lambda $ are given in \eqref{DefOmega}, \eqref{Def:Lambda} respectively, and
\begin{align}
N&:=z(W+sU)+L\label{def:N}\\
	W&:=\frac{1}{\phi}\left\{\frac{\varphi}{2}-s\phi U - \phi_zL - (r^2-s^2)\phi_sU\right\},\label{def:W}\\
    L&:=\frac{1}{2 \Lambda}\left\{-(r^2-s^2)\left(\varphi_s-\frac{2}{r}\phi_r\right)\phi_{sz} + \left(\varphi_z-2\phi_{x^0}\right)(\Omega+(r^2-s^2)\phi_{ss})\right\},\label{def:L}\\
 U&:=\frac{1}{2\Lambda}\left\{\left(\varphi_s-\frac{2}{r}\phi_r\right)\phi_{zz}-\left(\varphi_z-2\phi_{x^0}\right)\phi_{sz}\right\},\label{def:U}\\
	V&:=\frac{1}{2\Lambda}\left\{\left(\varphi_s-\frac{2}{r}\phi_r\right)\phi_{sz}-\left(\varphi_z-2\phi_{x^0}\right)\phi_{ss}\right\},\nonumber\\
	\varphi &:=z\phi_{x^0}+\frac{s}{r}\phi_r+\phi_s.\nonumber
    \end{align}
If we consider $\phi=\phi(x^0,r,z),$ we have
\end{proposition}
\begin{corollary}\label{cor:phi_no_s}\normalfont
Let $F=\vert\overline{y}\vert\,\phi(x^0,r,z)$ be a Finsler metric defined on $M$, where 
$z=\frac{y^0}{\vert\overline{y}\vert}$, $r=\vert\overline{x}\vert$,
$s=\frac{\langle\overline{x},\overline{y}\rangle}{\vert\overline{y}\vert}$ and $TM$ is endowed with coordinates \eqref{coordx}--\eqref{coordy}. Then the geodesic spray coefficients $G^A$ are given by
\begin{align*}
G^0&=u^2N,\\
G^i&=u^2Wu_i+u^2Ux^i,
\end{align*}
where $u=\vert\overline{y}\vert$ and $u_i=\frac{\partial u}{\partial y^i}=\frac{y^i}{u}$, and,
\begin{align*}
W &:= \frac{1}{\phi}\left\{\frac{z}{2}\phi_{x^0}-\frac{s}{2r}\phi_r-s\phi U-\phi_z L\right\},\\
U &:= -\frac{1}{2r}\frac{\phi_r}{\phi-z\phi_z},
\\
L &:= \frac{1}{2\phi_{zz}}\left\{z\phi_{x^0z}+\frac{s}{r}\phi_{rz}-\phi_{x^0}\right\},\\
N&:=z(W+sU)+L.
\end{align*}
\end{corollary}

To compute the Berwald and Landsberg curvatures of weakly orthogonally invariant Finsler metrics \eqref{def:F}, we adopt the operator
\[
\Psi(\Theta) = -s \Theta_s - z \Theta_z,
\]
for any differentiable function $\Theta = \Theta(s,z)$.
This notation was introduced in \cite{SolorzanoLujerioLeonRodriguez2026} (see Remark 3), where several of its structural properties were established.

For convenience, we collect below some identities involving $\Psi$ that will be used throughout the computations:
\begin{align}
\Psi\left(z^2\Psi\left(\dfrac{\Theta}{z^2}\right)\right)&=-sz\Psi\left(\dfrac{\Theta_s}{z}\right)-z^2\Psi\left(\dfrac{\Theta_z}{z}\right),\label{psieq1}\\
\dfrac{1}{z}\Psi\left(z^2\Psi\left(\dfrac{\Theta}{z}\right)\right)&=-s\Psi(\Theta_s)-z\Psi(\Theta_z)-z\Psi\left(\dfrac{\Theta}{z}\right),\label{eq:betaa1}\\
\frac{1}{z^2}\Psi(z^2\Psi(\Theta)) &= -3\Psi(\Theta) - s\Psi(\Theta_s) - z\Psi(\Theta_z),\label{eq:betaa2}\\
\Psi\left(\Theta_z\right)&=\Psi_z(\Theta) + \Theta_z,\label{psieq4}\\
z\Psi_z(\Theta)&=\Psi\left(z\Theta_z\right),\label{eq:psiztheta}\\
z\Psi_s\left(\frac{\Theta}{z}\right) &= \Psi (\Theta_s),\label{psieq2}\\
\left(z\Psi\left(\frac{\Theta}{z}\right)\right)_z &= \Psi (\Theta_z),\label{psieq3}
\end{align}

Moreover, the following higher-order identities will also be required:
\begin{align}\label{eq:psi3I}
\frac{1}{z^2}\Psi\left(z^2\Psi\left(z^2\Psi\left(\frac{\Theta}{z^2}\right)\right)\right)&=- 2 \Psi\left(z^2\Psi\left(\frac{\Theta}{z^2}\right)\right) -\frac{s}{z}\Psi\left(z^2\Psi\left(\frac{\Theta_s}{z}\right)\right)  - \Psi\left(z^2\Psi\left(\frac{\Theta_z}{z}\right)\right),
\end{align}

\begin{align}\label{eq:psi3II}
\frac{1}{z^2}\Psi\left(z^2\Psi\left(z^2\Psi\left(\frac{\Theta}{z}\right)\right)\right)&=- 3 \Psi\left(z^2\Psi\left(\frac{\Theta}{z}\right)\right) -\frac{s}{z}\Psi\left(z^2\Psi\left(\Theta_s\right)\right)  - \Psi\left(z^2\Psi\left(\Theta_z\right)\right).
\end{align}
Whit this,  
\begin{align}
	\frac{\partial\Theta}{\partial y^0}&= \frac{\Theta_z}{u},\label{partial0}
	\\
	u\frac{\partial\Theta}{\partial y^l}&=\Theta_sx^l +\Psi(\Theta)u_l,
	\label{thetal}\\
 u\frac{\partial }{\partial y^k } \left(\Theta u_ l\right)&=\Theta\delta_{kl} + \Theta_sx^ku_l+ \frac{1}{z}\Psi\left(z\Theta\right)u_ku_l,\label{eq:Thetaul}\\
     u\frac{\partial }{\partial y^j}\left(\Theta u_ku_l\right)&=\Theta\left(\delta_{jk}u_l \right)_{\overrightarrow{kl}} + \Theta_sx^ju_ku_l + \frac{1}{z^2}\Psi\left(z^2\Theta\right)u_ju_ku_l,\label{eq:Thetaukul}\\
 u\frac{\partial }{\partial y^j}\left(\Theta u_ku_lu_i\right)&=\Theta (\delta_{jk}u_lu_i )_{\overrightarrow{kli}} + \Theta_sx^ju_ku_lu_i + \frac{1}{z^3}\Psi\left(z^3\Theta\right)u_ju_ku_lu_i, \label{eq:tetajkli}
\end{align}
where $(.)_{jkl}$ denotes the cyclic permutation (ex.: $(\delta_{jk}u_lu_i )_{\overrightarrow{kli}}=\delta_{jk}u_lu_i+\delta_{jl}u_iu_k + \delta_{ji}u_ku_l$).
\section{Berwald and Landsberg curvatures}

The Berwald tensor $ \mathcal{B}=\mathcal{B}^{A}_{BCD}\partial_A\otimes dx^B\otimes dx^C\otimes dx^D $ is defined by
\begin{align*}
	\mathcal{B}^{A}_{BCD}:=&\frac{\partial^3 G^A}{\partial y^B\partial y^C\partial y^D}.
\end{align*}
A Finsler metric is called a Berwald metric if $ \mathcal{B}^A_{BCD}=0, $ i.e. the spray coefficients $ G^A=G^A(x,y)$ are quadratic in $ y\in T_x{M} $ at every point $ x\in M. $

The Landsberg tensor $L=L_{ABC}dx^A\otimes dx^B\otimes dx^C$ is defined by

\begin{align*}
	L_{ABC}:=-\frac{1}{4}[F^2]_{y^D}[G^D]_{y^Ay^By^C}.
\end{align*}
A Finsler metric is called a \textit{Landsberg metric} if $ L_{ABC}=0. $ 

By the positive homogeneity of $F$ we have
\[L_{ABC}=\frac{1}{2}FF_{y^D}\mathcal{B}^D_{ABC}=\frac{1}{2}F\left(F_{y^0}\mathcal{B}^0_{ABC} + \sum_i F_{y^i}\mathcal{B}^i_{ABC}\right).\]

\begin{theorem}\label{Bcurv}\normalfont Let $F=\vert\overline{y}\vert{\phi(x^0,r,s,z)}$ be a  Finsler metric defined on $ M $, where $ z=\frac{y^0}{\vert\overline{y}\vert}, $ $r=\vert\overline{x}\vert$, $ s=\frac{\langle\overline{x},\overline{y}\rangle}{\vert\overline{y}\vert} $ and  $TM $  with coordinates   \eqref{coordx}, \eqref{coordy}.  Then the
Berwald curvature of $F$ is given by
\begin{align*}
	\mathcal{B}^0_{000}&=\frac{1}{u}N_{zzz},\\
	\mathcal{B}^0_{00l}&=\frac{1}{u}\left\{N_{szz}x^l+\Psi(N_{zz})u_l\right\},\\
	\mathcal{B}^{0}_{0kl}&=\frac{1}{u}\left\{N_{ssz}x^kx^l + \Psi\left(N_{sz}\right)(x^lu_k)_{\overrightarrow{lk}} + z\Psi\left(\frac{N_z}{z}\right)\delta_{kl} + \frac{1}{z}\Psi\left(z^2\Psi\left(\frac{N_z}{z}\right)\right)u_ku_l \right\},\\
\mathcal{B}^0_{jkl}&=\frac{1}{u}\left\{\frac{N_{sss}}{3}x^jx^kx^l+\Psi(N_{ss})x^jx^ku_l +z\Psi\left(\frac{N_s}{z}\right)x^j\delta_{kl}+\Psi\left(z^2\Psi\left(\frac{N}{z^2}\right)\right)u_j\delta_{kl}  \right.\\
    &\qquad\left.+\frac{1}{z}\Psi\left(z^2\Psi\left(\frac{N_s}{z}\right)\right)x^ju_ku_l+ \frac{1}{3z^2}\Psi\left(z^2\Psi\left(z^2\Psi\left(\frac{N}{z^2}\right)\right)\right)u_ju_ku_l\right\}_{\overrightarrow{jkl}},\\
\mathcal{B}^i_{000}&=\frac{1}{u}\left\{U_{zzz}x^i + W_{zzz}u_i\right\},\\
\mathcal{B}^i_{00l}&=\frac{1}{u}\left\{U_{szz}x^lx^i + \Psi\left(U_{zz}\right)x^iu_l + W_{zz}\delta_{il} + W_{szz}x^lu_ i + \Psi_z\left(W_{z}\right)u_lu_i\right\},\\
\mathcal{B}^i_{0kl}&= \frac{1}{u}\left\{U_{ssz}x^kx^lx^i+z\Psi\left(\frac{U_z}{z}\right)\delta_{kl}x^i+ \frac{1}{z}\Psi\left(z^2\Psi\left(\frac{U_z}{z}\right)\right)u_ku_lx^i\right.\\
&\left.\quad\qquad +W_{ssz}x^kx^lu_i + \frac{1}{z^2}\Psi\left(z^2\Psi\left(W_z\right)\right)u_k u_lu_i \right\}\\
    &\quad +\frac{1}{u}\left\{\Psi(U_{sz})u_kx^lx^i  + W_{sz}x^k\delta_{li} +\frac{1}{z}\Psi\left(zW_{sz}\right)x^lu_ku_i \right\}_{\overrightarrow{kl}}+ \frac{1}{u}\Psi(W_z)(\delta_{il}u_k)_{\overrightarrow{ikl}},\\
\mathcal{B}^i_{jkl}&= \frac{1}{u}\left\{U_{sss}x^jx^kx^l+\frac{1}{z^2}\Psi\left(z^2\Psi\left(z^2\Psi\left(\frac{U}{z^2}\right)\right)\right)u_ju_ku_l\right.\\
 &\quad+\left. \left[\Psi(U_{ss})u_jx^kx^l+z\Psi\left(\frac{U_s}{z}\right)\delta_{jk}x^l+\Psi_s\left(z^2\Psi\left(\frac{U}{z^2}\right)\right)u_ju_kx^l  
 +\Psi\left(z^2\Psi\left(\frac{U}{z^2}\right)\right) \delta_{jk}u_l\right]_{\overrightarrow{jkl}}\right\}x^i
 \\
&\quad +\frac{1}{u}\left\{\left[W_{ss}\delta_{ij}x^kx^l+\Psi_s(W_s)u_iu_jx^kx^l + \frac{1}{z^2}\Psi\left(z^2\Psi\left(W_s\right)\right)x^ju_ku_lu_i\right]_{\overrightarrow{jkl}}\right.\\
&\quad +z\Psi\left(\frac{W}{z}\right)(\delta_{ji}\delta_{kl})_{\overrightarrow{ikl}} + \Psi\left(W_s\right)\left(x^j(u_i\delta_{kl})_{\overrightarrow{ikl}} + x^k(u_j\delta_{il})_{\overrightarrow{ijl}} + x^l(u_i\delta_{jk})_{\overrightarrow{ijk}}\right)\\
&\quad \left.+\frac{1}{z}\Psi\left(z^2\Psi\left(\frac{W}{z}\right)\right)(\delta_{ji}u_ku_l + \delta_{ik}u_lu_j)_{\overrightarrow{ikl}}+W_{sss}x^jx^kx^lu_i +\frac{1}{z^3}\Psi\left(z^2\Psi\left(z^2\Psi\left(\frac{W}{z}\right)\right)\right)u_ju_ku_lu_i\right\},
\end{align*}
where   $u=\vert\overline{y} \vert$, $u_i=\frac{\partial u}{\partial y^i}=\frac{y^i}{u}$, and $N$, $W$ and $U$ are in \eqref{def:N}, \eqref{def:W} and \eqref{def:U} respectively, and $(.)_{jkl}$ denotes the cyclic permutation.
\end{theorem}

\begin{proof}
 By Proposition \ref{prop2} and from \eqref{partial0}, \eqref{thetal}, \eqref{eq:Thetaul}, \eqref{eq:Thetaukul}, we have,
\begin{align*}
    \mathcal{B}^0_{000}=&\frac{\partial^3}{\partial y^0\partial y^0\partial y^0}(u^2N)=\frac{\partial^2}{\partial y^0\partial y^0}(u^2\frac{N_z}{u})=\frac{\partial}{\partial y^0}(N_{zz}) = \frac{N_{zzz}}{u},\\
    \mathcal{B}^0_{00l}=&\frac{\partial^3}{\partial y^0\partial y^0\partial y^l}(u^2N)=\frac{\partial}{\partial y^l}(N_{zz})=\frac{1}{u}\left[N_{szz}x^l + \Psi(N_{zz})u_l\right],
\end{align*}
and using the identity $\Psi(N_{sz}) = z\Psi_s\left(\frac{N_z}{z}\right)$, where the sub index $s$ represents the partial derivative in $s$, we obtain
\begin{align*}
    \mathcal{B}^0_{0kl}&=\frac{\partial^2}{\partial y^k\partial y^l}\left(uN_z\right)=y^0\frac{\partial^2}{\partial y^k\partial y^l}\left(\frac{N_z}{z}\right)
    = \frac{\partial }{\partial y^k}\left({N_{sz}}x^l +{z} \Psi\left(\frac{N_z}{z}\right)u_l\right)\\
    &=\frac{1}{u}\left\{N_{ssz}x^kx^l + \Psi\left(N_{sz}\right)(x^lu_k)_{\overrightarrow{lk}} + z\Psi\left(\frac{N_z}{z}\right)\delta_{kl} + \frac{1}{z}\Psi\left(z^2\Psi\left(\frac{N_z}{z}\right)\right)u_ku_l \right\},\\
    \mathcal{B}^0_{jkl}&=\frac{\partial^3}{\partial y^j\partial y^k\partial y^l}\left(u^2N\right)=\frac{\partial^2}{\partial y^j\partial y^k}\left(uN_sx^l + z^2\Psi\left(\frac{N}{z^2}\right)uu_l\right)\\
    &=\frac{\partial }{\partial y ^j}\left[N_{ss}x^kx^l + z\Psi\left(\frac{N_s}{z}\right)(x^lu_k)_{\overrightarrow{kl}} + z^2\Psi\left(\frac{N}{z^2}\right)\delta_{kl} + \Psi\left(z^2\Psi\left(\frac{N}{z^2}\right)\right)u_ku_l\right].
\end{align*}
From \eqref{eq:G^i},  \eqref{thetal}, \eqref{eq:Thetaul} and \eqref{eq:psiztheta}, $\mathcal{B}^i_{000}$ and $\mathcal{B}^i_{00l}$ are directly obtained. Using the properties of $\Psi $ we have,
\begin{align*}
    \mathcal{B}^i{}_{0kl}&=\frac{\partial^2}{\partial y^k\partial y^l}\left(uU_zx^i +uW_zu_i\right)\\
    &=\frac{\partial }{\partial y^k} \left(U_{sz}x^lx^i+z\Psi\left(\frac{U_z}{z}\right)u_lx^i + W_z\delta_{li} + W_{sz}x^lu_i + \Psi(W_z)u_lu_i\right).
\end{align*}
Analogous to the previous cases, using \eqref{eq:tetajkli}, we have
\begin{align*}
\mathcal{B}^i_{jkl}&=\frac{\partial^2}{\partial y^j\partial y^k}\left(2Uuu_lx^i+uU_sx^lx^i+u\Psi(U)u_lx^i+ 2uWu_lu_i+uW\delta_{li}+ uW_sx^lu_i+ \dfrac{\Psi(zW)}{z}uu_lu_i\right)\\
    &=\frac{\partial}{\partial y^j}\left\{\Psi\left(z^2\Psi\left(\dfrac{U}{z^2}\right) \right)u_ku_lx^i+z^2\Psi\left(\dfrac{U}{z^2}\right)\delta_{kl}x^i+z\Psi\left(\dfrac{U_s}{z}\right)(x^ku_l)_{ \overrightarrow{kl}}x^i\right.\\
    &\quad \left.+\dfrac{1}{z}\Psi\left(z^2\Psi\left(\dfrac{W}{z}\right)\right) u_ku_lu_i+ z\Psi\left(\dfrac{W}{z}\right)(\delta_{kl}u_i)_{ \overrightarrow{kli}}+\Psi(W_s)(x^ku_l)_{ \overrightarrow{kl}}u_i+U_{ss}x^lx^kx^i\right.\\
    &\quad\left. +W_{ss}u_ix^kx^l+W_s(x^k\delta_{li})_{ \overrightarrow{kl}}\dfrac{}{}\right\rbrace.\end{align*}

\end{proof}

\begin{lemma}\label{maintheo1}\normalfont
    Let $F=u \phi(x^0,r,s,z),$ be a Finsler metric defined on $I\times \mathbb{B}^n(\rho)$, $n\geq 3$, where $z=\frac{y^0}{u},$ $r=\vert \overline{x}\vert$, $s=\frac{\langle \overline{x},\overline{y}\rangle}{u}$, $u=\vert \overline{y}\vert$, and $TM$ defined with coordinates \eqref{coordx}, \eqref{coordy}. Then $F$ has vanishing Berwald curvature if, and only if, $\phi$ satisfies
    \begin{align}
        (a)\quad z\Psi\left(\frac{U_s}{z}\right)&=0 &(b)\quad z\Psi\left(\frac{U_z}{z}\right)&=0, &(c)\quad U_{zzz}&=0,\label{eq:pdeU}\\
        (a)\quad z\Psi\left(\frac{N_s}{z}\right)&=0 & (b)\quad z\Psi\left(\frac{N_z}{z}\right)&=0, &(c)\quad N_{zzz}&=0,\label{eq:pdeR}\\
        (a)\quad z\Psi\left(\frac{W}{z}\right)&=0, &(b)\quad W_{zz}&=0,\label{eq:pdeT}
        \end{align}
    where  $N$, $W$ and $U$ are in \eqref{def:N}, \eqref{def:W} and \eqref{def:U} respectively.
\end{lemma}

\begin{proof} Suppose $F$ has vanishing Berwald curvature. Consider the orthonormal matrix $O\in O(n)$ (See the proof of Proposition 1.3.1 in \cite{GuoMo2018Book} or the proof of Lemma 1 in \cite{Solorzano2022}) such that 
   \begin{align*}
       \tilde{x}&=O\overline{x}=\left(\vert \overline{x}\vert, 0,\ldots,0\right)\\
       \Tilde{y}&=O\overline{y}=\left(\frac{\langle \overline{x}, \overline{y}\rangle}{\vert \overline{x}\vert}, \frac{\sqrt{\vert \overline{x}\vert^2\vert \overline{y}\vert^2 - \langle \overline{x},\overline{y}\rangle^2}}{\vert \overline{x}\vert},0,\ldots,0\right).
   \end{align*} 
   For the invariance of $r,s$ and $z$ under the action $O$. Then, from $\mathcal{B}^0_{000}=0$, we obtain $N_{zzz}=0$. From $\mathcal{B}^0_{033}=0$, we get \begin{equation}\label{Deq1}
z\Psi\left(\dfrac{N_z}{z}\right)=0.
   \end{equation}
   Using property \eqref{psieq1} and \eqref{Deq1}, we have
\begin{equation}\label{Deq2}
       \Psi\left(z^2\Psi\left(\dfrac{N}{z^2}\right)\right)=-sz\Psi\left(\dfrac{N_s}{z}\right).
   \end{equation}
   From $\mathcal{B}^0_{133}=0$, we obtain
   \begin{equation}\label{Deq3}    
   rz\Psi\left(\dfrac{N_s}{z}\right)+\dfrac{s}{r}\Psi\left(z^2\Psi\left(\dfrac{N}{z^2}\right)\right)=0.
   \end{equation}
   Substituting \eqref{Deq2} into \eqref{Deq3}, we get
   \[\left(\dfrac{r^2-s^2}{r}\right)\left(z\Psi\left(\dfrac{N_s}{z}\right)\right)=0.\]Hence, \[z\Psi\left(\dfrac{N_s}{z}\right)=0.
\]Thus, \eqref{eq:pdeR} is satisfied.

From $\mathcal{B}^1_{000}=0$, we get  
\begin{equation}\label{Deqnw1}
 rU_{zzz}+\dfrac{s}{r}W_{zzz}=0.   
\end{equation}
Also, from $\mathcal{B}^2_{000}=0$ and \eqref{Deqnw1}, we obtain $U_{zzz}=0$. 

From $\mathcal{B}^1_{033}=0$, we have 
\[z\Psi\left(\dfrac{U_z}{z}\right)=0.\]
From $\mathcal{B}^3_{003}=0$, 
 we have $W_{zz}=0$. Also, from $\mathcal{B}^3_{333}=0$, we obtain 
 \begin{equation}\label{Deqnw2}
 z\Psi\left(\dfrac{W}{z}\right)=0\text{ (this implies }\Psi(W_s)=0\text{)}.    
 \end{equation}
 Thus, \eqref{eq:pdeT} is satisfied.
 
From $\mathcal{B}^1_{331}=0$ and \eqref{Deqnw2}, we obtain \[rz\Psi\left(\dfrac{U_s}{z}\right)+\dfrac{s}{r}\Psi\left(z^2\Psi\left(\dfrac{U}{z^2}\right)\right)=0.\]Similarly, as in the case of $N$, we conclude
   \[\left(\dfrac{r^2-s^2}{r}\right)\left(z\Psi\left(\dfrac{U_s}{z}\right)\right)=0.\]Therefore, \[z\Psi\left(\dfrac{U_s}{z}\right)=0,
\]and thus \eqref{eq:pdeU} is satisfied.

Conversely, assume that $\phi$ satisfies \eqref{eq:pdeU}, \eqref{eq:pdeR} and \eqref{eq:pdeT}. From \eqref{eq:pdeR} $(c)$, we get $\mathcal{B}^0_{000}=0$. 

By property \eqref{psieq3} and \eqref{eq:pdeR} $(b)$, we have
\begin{equation}\label{Deq6}
\Psi(N_{zz})=\left(z\Psi\left(\dfrac{N_z}{z}\right)\right)_z=0.
\end{equation}
From \eqref{Deq6} and \eqref{eq:pdeR} $(c)$, we get
\begin{equation}\label{Deq7}
    N_{zzs}=0.
\end{equation}
Therefore, by \eqref{Deq6} and \eqref{Deq7}, we obtain $\mathcal{B}^0_{00l}=0$.   Using property \eqref{psieq2} and \eqref{eq:pdeR} $(a),\;(b)$, we obtain
\begin{equation}\label{Deq4} \Psi(N_{ss})=z\Psi_s\left(\dfrac{N_s}{z}\right)=0,
\end{equation}
and 
\begin{equation}\label{Deq5} \Psi(N_{sz})=z\Psi_s\left(\dfrac{N_z}{z}\right)=0.
\end{equation}
From \eqref{Deq5} and \eqref{Deq7}, we also have
\begin{equation}\label{Deq8}
 N_{zss}=0.   
\end{equation}
Consequently, by \eqref{Deq8}, \eqref{Deq5} and \eqref{eq:pdeR} $(b)$, we have $\mathcal{B}^0_{0kl}=0$. From \eqref{Deq4} and \eqref{Deq8}, we obtain
\begin{equation}\label{Deq9}
    N_{sss}=0.
\end{equation}
Also, by property \eqref{psieq1} and \eqref{eq:pdeR} $(a),\;(b)$, we have
\begin{equation}\label{Deq10}
\Psi\left(z^2\Psi\left(\dfrac{N}{z^2}\right)\right)=-sz\Psi\left(\dfrac{N_s}{z}\right)-z^2\Psi\left(\dfrac{N_z}{z}\right)=0. 
\end{equation}
Thus, by \eqref{Deq9}, \eqref{Deq4}, \eqref{eq:pdeR} $(a)$ and \eqref{Deq10}, we conclude that $\mathcal{B}^0_{jkl}=0$. From \eqref{eq:pdeU} $(c)$ and \eqref{eq:pdeT} $(b)$, we have $\mathcal{B}^i_{000}=0$. Now, by \eqref{eq:pdeU}, analogously as in the case of $N$ we obtain
\begin{align}
\Psi(U_{ss})=\Psi(U_{sz})=\Psi(U_{zz})=0\label{Deq11},\\
U_{zss}=U_{zzs}=U_{sss}=0,\quad \Psi\left(z^2\Psi\left(\dfrac{U}{z^2}\right)\right)=0.\label{Deq12}
\end{align}
On the other hand, by property \eqref{psieq4} and \eqref{eq:pdeT} $(b)$, we have
\begin{equation}\label{Deq13}
\Psi_z(W_z)=\Psi(W_{zz})-W_{zz}=0.    
\end{equation}
Therefore, by \eqref{Deq11}, \eqref{Deq12}, \eqref{eq:pdeT} $(b)$ and \eqref{Deq13}, we obtain $\mathcal{B}^i_{00l}=0$. By property \eqref{psieq3} and \eqref{eq:pdeT} $(a)$, we get
\begin{equation}\label{Deq14}
\Psi(W_z)=\left(z\Psi\left(\dfrac{W}{z}\right)\right)_z=0.
\end{equation}From \eqref{Deq14} and \eqref{eq:pdeT} $(b)$, we obtain 
\begin{equation}\label{Deq15}
    W_{sz}=0.
\end{equation}
Consequently, by \eqref{Deq11}, \eqref{Deq12}, \eqref{eq:pdeU} $(b)$, \eqref{Deq14} and \eqref{Deq15}, we get $\mathcal{B}^i_{0kl}=0$. By property \eqref{psieq2} and \eqref{eq:pdeT} $(a)$, we obtain
\begin{equation}\label{Deq16} \Psi(W_s)=z\Psi_s\left(\dfrac{W}{z}\right)=0.
\end{equation} From \eqref{Deq16} and \eqref{Deq15}, we have
\begin{equation}\label{Deq17}
    W_{ss}=0.
\end{equation}
Therefore, by \eqref{Deq11}, \eqref{Deq12}, \eqref{eq:pdeU}, \eqref{eq:pdeT} $(a)$, \eqref{Deq16} and \eqref{Deq17}, we obtain $\mathcal{B}^i_{jkl}=0$.
\end{proof}

\begin{theorem}\label{Thm2:B0}\normalfont
    Let $F= u \phi(x^0,r,s,z),$ be a Finsler metric defined on $I\times \mathbb{B}^n(\rho)$, $n\geq 3$, where $z=\frac{y^0}{u},$ $r=\vert \overline{x}\vert$, $s=\frac{\langle \overline{x},\overline{y}\rangle}{u}$, $u=\vert \overline{y}\vert$ and $TM$ defined with coordinates \eqref{coordx}, \eqref{coordy}. Then $F$ has vanishing Berwald curvature if, and only if, the positive function $\phi$ satisfies
    \begin{align}
z\phi_{x^0s} - \frac{1}{r}\phi_r + \frac{s}{r}\phi_{rs}+\phi_{ss}&=2\left[\left(\phi-s\phi_s- z\phi_z+(r^2-s^2)\phi_{ss}\right)U +\phi_{sz}L\right],\label{eq:p1AU}\\
   z\phi_{x^0z} - \phi_{x^0} + \frac{s}{r}\phi_{rz} + \phi_{sz}&=2\left[(r^2-s^2)\phi_{sz}U+ \phi_{zz}L\right],\label{eq:p2AU}\\
   z\phi_{x^0} +\frac{s}{r}\phi_r + \phi_s &= 2\left[W\phi + (s\phi +(r^2-s^2)\phi_s)U + \phi_z L\right]\label{eq:third}
\end{align}
where
    \begin{align*}
U&=f_1\frac{s^2}{2}+f_2 sz + f_3\frac{z^2}{2} + f_4, \\
L&= f_5\frac{s^2}{2} +f_6sz +f_7\frac{z^2}{2}+f_8 -sz\left(f_1\frac{s^2}{2}+f_2 sz + f_3\frac{z^2}{2}\right),\\
W&=f_9 s + f_{10}z
\end{align*}
  and $f_i=f_i(x^0,r), \, i=1,...,10,$ are arbitrary differentiable functions. 
\end{theorem}

\begin{proof}
 From \eqref{eq:pdeU}, \eqref{eq:pdeR} and \eqref{eq:pdeT} we have that there are differentiable functions $g_i=g_i(x^0,r)$ such that
    \begin{align*}
        U&=g_1\frac{s^2}{2}+g_2 sz + g_3\frac{z^2}{2} + g_4,\\
        N&=g_5\frac{s^2}{2}+g_6 sz + g_7\frac{z^2}{2} + g_8,\\
        W&=g_9s + g_{10}z.
    \end{align*}
From \eqref{def:N}, \[L=N-z(W+sU).\]

Equation \eqref{eq:third} become from the definition of $W$ in \eqref{def:W}. From definition of $U$ and $L$ in \eqref{def:U} and \eqref{def:L}, we have
\begin{align}
    \phi_{zz}p_1-\phi_{sz}p_2&=2\Lambda U, \label{eq:UA01}\\
    -(r^2-s^2)\phi_{sz}p_1+(\Omega +(r^2-s^2)\phi_{ss})p_2&=2\Lambda L, \label{eq:UA02}
\end{align}
where
\begin{align}
   p_1&:=\varphi_s-\frac{2}{r}\phi_r=z\phi_{x^0s} - \frac{1}{r}\phi_r + \frac{s}{r}\phi_{rs}+\phi_{ss},\label{def:p1}\\
   p_2&:=\varphi_z-2\phi_{x^0}=z\phi_{x^0z} - \phi_{x^0} + \frac{s}{r}\phi_{rz} + \phi_{sz},\label{def:p2}
\end{align}
and $\Omega, \Lambda$ are given in \eqref{DefOmega} and \eqref{Def:Lambda} respectively. Due to the fact $\Lambda\neq 0, $ the system \eqref{eq:UA01}- \eqref{eq:UA02} is equivalent to the system \eqref{eq:p1AU}-\eqref{eq:p2AU}. 

On the other hand, suppose that 
\begin{align*}
        U&=f_1\frac{s^2}{2}+f_2 sz + f_3\frac{z^2}{2} + f_4,\\
        L&= f_5\frac{s^2}{2} +f_6sz +f_7\frac{z^2}{2}+f_8 -sz\left(f_1\frac{s^2}{2}+f_2 sz + f_3\frac{z^2}{2}\right),\\
        W&= f_9 s + f_{10}z
    \end{align*}
then, from definition of $N$  in \eqref{def:N}  we have
\begin{align*}
N&=z(W+sU) + L\\
&= \frac{f_5}{2}s^2
+ \left(f_9+f_4+f_6\right)s z
+ \left(f_{10}+\frac{f_7}{2}\right)z^2
+ f_8 .
\end{align*}
Thus, from Lemma \ref{maintheo1}, we conclude the proof.
\end{proof}

\begin{theorem}\label{Lcurv}\normalfont
    With the same notation of Theorem \ref{maintheo1}. The Landsberg curvature of $F=u\phi$ is given by $L_{ABC}=\frac{1}{2}\phi \overline{L}_{ABC}$, where
    \begin{align*}
        \overline{L}_{000}&=\phi_zN_{zzz} + \left(r^2\phi_s + s\Omega\right)U_{zzz} + \left(s\phi_s + \Omega\right)W_{zzz},
    \\
        \overline{L}_{00l}&=\left(\phi_zN_{szz} + (r^2\phi_s+s\Omega)U_{szz}  + (s\phi_s+\Omega)W_{szz}+ \phi_s W_{zz}\right)x^l \\
        &\qquad +\left(\phi_z\Psi\left(N_{zz}\right) + (r^2\phi_s +s\Omega)\Psi(U_{zz}) + (s\phi_s+\Omega)\Psi(W_{zz}) -s\phi_s W_{zz}  \right)u_l,
    \\
        \overline{L}_{0kl}&= \left[\phi_zN_{ssz} + (r^2\phi_s +s\Omega)U_{ssz} + (s\phi_s+\Omega)W_{ssz}+ 2\phi_s W_{sz} \right]x^kx^l \\
        &\qquad+ \left[z\phi_z\Psi\left(\frac{N_z}{z}\right)+z(r^2\phi_s+s\Omega)\Psi\left(\frac{U_z}{z}\right) + (s\phi_s+\Omega)\Psi(W_z) \right]\delta_{kl}+\nonumber\\
        &\quad +\left\{\frac{1}{z}\phi_z \Psi\left(z^2\Psi\left(\frac{N_z}{z}\right)\right) +\frac{r^2\phi_s + s\Omega}{z}\Psi\left(z^2\Psi\left(\frac{U_z}{z}\right)\right) + \frac{s\phi_s + \Omega}{z^2}\Psi\left(z^2\Psi\left(W_z\right)\right)+ 2\Omega\Psi(W_z)\right\}u_ku_l\\
        &\quad +\left\{\phi_z\Psi\left(N_{sz}\right)+(r^2\phi_s + s\Omega)\Psi\left(U_{sz}\right) + {(s\phi_s + \Omega)}\Psi\left(W_{sz}\right)+\phi_s \Psi(W_{z}) -s\phi_s W_{sz}\right\}\left(x^ku_l\right)_{\overrightarrow{kl}},
    \\
        \overline{L}_{jkl}&=\left\{ \frac{1}{3}\phi_zN_{sss}+ \frac{(r^2\phi_s+s\Omega)}{3}U_{sss}+ \frac{s\phi_s+\Omega}{3}W_{sss} +      \phi_s W_{ss}\right\}(x^jx^kx^l)_{\overrightarrow{jkl}}\\
        &\quad + \left\{\phi_z\Psi(N_{ss})+(r^2\phi_s +s\Omega)\Psi\left(U_{ss}\right) + (s\phi_s + \Omega)\Psi(W_{ss}) +2\phi_s\Psi(W_s)  -s\phi_sW_{ss}\right\}(u_jx^kx^l)_{\overrightarrow{jkl}}\\
        &\quad + \left\{\frac{\phi_z}{3z^2}\Psi\left(z^2\Psi\left(z^2\Psi\left(\frac{N}{z^2}\right)\right)\right)+\frac{r^2\phi_s +s\Omega }{3z^2}\Psi\left(z^2\Psi\left(z^2\Psi\left(\frac{U}{z^2}\right)\right)\right) \right.
        \\
        &\qquad\qquad\qquad\qquad  \left.+ \frac{s\phi_s + \Omega}{3z^3}\Psi\left(z^2\Psi\left(z^2\Psi\left(\frac{W}{z}\right)\right)\right) + \frac{\Omega}{z}\Psi\left(z^2\Psi\left(\frac{W}{z}\right)\right)\right\}(u_ju_ku_l)_{\overrightarrow{jkl}}
        \\
        &\quad +\left\{z\phi_z\Psi\left(\frac{N_s}{z}\right)+z(r^2\phi_s + s\Omega)\Psi\left(\frac{U_s}{z}\right) + (s\phi_s + \Omega)\Psi\left(W_s\right) + z\phi_s\Psi\left(\frac{W}{z}\right)\right\}(\delta_{jk}x^l)_{\overrightarrow{jkl}}\\
        &\quad +\left\{\frac{\phi_z}{z}\Psi\left(z^2\Psi\left(\frac{N_s}{z}\right)\right)+\frac{(r^2\phi_s +s\Omega)}{z}\Psi\left(z^2\Psi\left(\frac{U_s}{z}\right)\right)+\frac{s\phi_s + \Omega}{z^2}\Psi\left(z^2\Psi\left(W_s\right)\right)\right.\\
        &\qquad\qquad\qquad  \left.+\frac{\phi_s}{z}\Psi\left(z^2\Psi\left(\frac{W}{z}\right)\right) + 2\Omega\Psi(W_s)\right\}(u_ju_kx^l)_{\overrightarrow{jkl}}\\
        &\quad +\left\{\phi_z\Psi\left(z^2\Psi\left(\frac{N}{z^2}\right)\right)+(r^2\phi_s + s\Omega)\Psi\left(z^2\Psi\left(\frac{U}{z^2}\right)\right)+ \frac{s\phi_s + \Omega}{z}\Psi\left(z^2\Psi\left(\frac{W}{z}\right)\right)\right.\\
        &\qquad\qquad\qquad  \left. +z\Omega\Psi\left(\frac{W}{z}\right)\right\}(\delta_{jk}u_l)_{\overrightarrow{jkl}}.
    \end{align*}
\end{theorem}

\begin{proof}
    From Theorem \ref{Bcurv} we have,

    \begin{align*}
    ux^i\mathcal{B}^i_{000}&=r^2U_{zzz}+sW_{zzz}
    \\
    uu^i\mathcal{B}^i_{000}&=sU_{zzz} + W_{zzz}
    \\
        ux^i\mathcal{B}^i_{00l}&=\left[r^2U_{szz} + W_{zz} + sW_{szz}\right]x^l + \left[r^2\Psi(U_{zz}) + s\Psi_z(W_{z})\right]u_l
    \\
        uu^i\mathcal{B}^i_{00l}&=\left[sU_{szz}+W_{szz}\right]x^l + \left[s\Psi(U_{zz}) + W_{zz} + \Psi_z(W_z)\right]u_l
    \\
        ux^i\mathcal{B}^i_{0kl}&=\left\{r^2U_{ssz} + sW_{ssz}+ 2W_{sz}\right\}x^kx^l + \left\{r^2z\Psi\left(\frac{U_z}{z}\right) + s\Psi(W_z) \right\}\delta_{kl}+\nonumber\\
        &\quad +\left\{\frac{r^2}{z}\Psi\left(z^2\Psi\left(\frac{U_z}{z}\right)\right) + \frac{s}{z^2}\Psi\left(z^2\Psi\left(W_z\right)\right)\right\}u_ku_l\\
        &\quad +\left\{r^2\Psi\left(U_{sz}\right) + \frac{s}{z}\Psi\left(zW_{sz}\right) + \Psi(W_z)\right\}\left(x^ku_l\right)_{\overrightarrow{kl}}
    \\
        uu^i\mathcal{B}^i_{0kl}&=\left\{sU_{ssz} + W_{ssz}\right\}x^kx^l\\
        &\quad +\left\{sz\Psi\left(\frac{U_z}{z}\right)+ \Psi\left(W_z\right)\right\}\delta_{kl}\\
        &\quad +\left\{\frac{s}{z}\Psi\left(z^2\Psi\left(\frac{U_z}{z}\right)\right) + \frac{1}{z^2}\Psi\left(z^2\Psi\left(W_z\right)\right)+ 2\Psi\left(W_z\right)\right\}u_ku_l\\
        &\quad +\left\{s\Psi(U_{sz}) +W_{sz} + \frac{1}{z}\Psi\left(zW_{sz}\right)\right\}(x^ku_l)_{\overrightarrow{kl}}
    \\
        ux^i\mathcal{B}^i_{jkl}&=\left\{r^2U_{sss} + W_{ss} + \frac{s}{3}W_{sss}\right\}(x^jx^kx^l)_{\overrightarrow{jkl}}\\
        &\quad + \left\{\frac{r^2}{z^2}\Psi\left(z^2\Psi\left(z^2\Psi\left(\frac{U}{z^2}\right)\right)\right) + \frac{s}{3z^3}\Psi\left(z^2\Psi\left(z^2\Psi\left(\frac{W}{z}\right)\right)\right)\right\}(u_ju_ku_l)_{\overrightarrow{jkl}}\\
        &\quad + \left\{r^2\Psi\left(U_{ss}\right) + s\Psi_s\left(W_s\right)+2\Psi\left(W_s\right)\right\}(u_jx^kx^l)_{\overrightarrow{jkl}}\\
        &\quad +\left\{r^2z\Psi\left(\frac{U_s}{z}\right) + z\Psi\left(\frac{W}{z}\right) + s\Psi\left(W_s\right)\right\}(\delta_{jk}x^l)_{\overrightarrow{jkl}}\\
        &\quad +\left\{r^2\Psi_s\left(z^2\Psi\left(\frac{U}{z^2}\right)\right) + \frac{s}{z^2}\Psi\left(z^2\Psi\left(W_s\right)\right)+ \frac{1}{z}\Psi\left(z^2\Psi\left(\frac{W}{z}\right)\right)\right\}(u_ju_kx^l)_{\overrightarrow{jkl}}\\
        &\quad +\left\{r^2\Psi\left(z^2\Psi\left(\frac{U}{z^2}\right)\right)+ \frac{s}{z}\Psi\left(z^2\Psi\left(\frac{W}{z}\right)\right)\right\}(\delta_{jk}u_l)_{\overrightarrow{jkl}}
    \end{align*}
    \begin{align*}
        uu^i\mathcal{B}^i_{jkl}&=\left\{\frac{s}{z^2}\Psi\left(z^2\Psi\left(z^2\Psi\left(\frac{U}{z^2}\right)\right)\right)+ \frac{1}{z}\Psi\left(z^2\Psi\left(\frac{W}{z}\right)\right)+ \frac{1}{3z^3}\Psi\left(z^2\Psi\left(z^2\Psi\left(\frac{W}{z}\right)\right)\right)\right\}(u_ju_ku_l)_{\overrightarrow{jkl}}\\
        &\quad +\left\{s\Psi\left(U_{ss}\right) + W_{ss}+\Psi_s\left(W_s\right)\right\}(u_jx^kx^l)_{\overrightarrow{jkl}} + \left\{sU_{sss}+ \frac{1}{3}W_{sss}\right\}(x^jx^kx^l)_{\overrightarrow{jkl}}\\
        &\quad + \left\{sz\Psi\left(\frac{U_s}{z}\right)+ \Psi\left(W_s\right)\right\}(\delta_{jk}x^l)_{\overrightarrow{jkl}}\\
        &\quad +\left\{s\Psi_s\left(z^2\Psi\left(\frac{U}{z^2}\right)\right)+\frac{1}{z^2}\Psi\left(z^2\Psi\left(W_s\right)\right) + 2\Psi\left(W_s\right)\right\}(u_ju_kx^l)_{\overrightarrow{jkl}}\\
        &+\left\{s\Psi\left(z^2\Psi\left(\frac{U}{z^2}\right)\right)+z\Psi\left(\frac{W}{s}\right)+\frac{1}{z}\Psi\left(z^2\Psi\left(\frac{W}{z}\right)\right)\right\}(\delta_{jk}u_l)_{\overrightarrow{jkl}}
    \end{align*}
     Using $F_{y^0}=\phi_z$, $F_{y^i}=\phi_sx^i + z\Psi\left(\frac{\phi}{z}\right)u_i=\phi_sx^i + \Omega u_i,$ and \[L_{BCD}=\frac{1}{2}F\left(F_{y^0}\mathcal{B}^0_{BCD} + F_{y^i}\mathcal{B}^i_{BCD}\right),\]
     we have the result.
\end{proof}

\section{On a generalized Finsler warped product metrics whit vanishing Lansberg curvature}

In this section, we assume that $\phi$ does not depend on $s$. This assumption is interesting not only because it simplifies the computations, but also because it allows us to generalize the family of Finsler warped metrics of the form $F = u\phi(x^0, z)$. 

It is knowed that every Landsberg Randers metrics $F=\alpha +\beta$ are Berwald metrics \cite{Matsumoto1974}. In the next theorem, which is a consequence of Theorem \ref{Thm2:B0}, we classify non-Randers Berwald metrics.

\begin{theorem}\label{cor2}\normalfont
      Let $F= u \phi(x^0,r,z),$ be a Finsler metric defined on $I\times \mathbb{B}^n(\rho)$, $n\geq 3$, where $z=\frac{y^0}{u},$ $r=\vert \overline{x}\vert$, $u=\vert \overline{y}\vert$ and $TM$ defined with coordinates \eqref{coordx}, \eqref{coordy}. Then $F$ has vanishing Berwald curvature if, and only if, one of the next two possibilities is satisfied
    \begin{enumerate}
        \item $F=u\phi$ is Berwald Randers metric, or
        \item $\phi$ satisfy
        \begin{align}
    -\frac{\phi_r}{r} &= 2(\phi-z\phi_z)g_1\label{eq:Berwald01}\\
    \phi_{x^0}&=z\phi_z g_2. \label{eq:Berwald02}
\end{align}
where $g_1=g_1(x^0,r), g_2=g_2(x^0,r)$ are arbitrary differentiable functions.
    \end{enumerate}
\end{theorem}
\begin{proof} From Theorem \ref{Thm2:B0} (or Theorem \ref{Bcurv} and Corollary \ref{cor:phi_no_s}) we have that $f_1=f_2=f_5=0$, and
    \begin{align}
- \frac{1}{r}\phi_r&=2\left(\phi- z\phi_z\right)(f_3\frac{z^2}{2} + f_4),\label{eq:B001}\\
\frac{1}{r}\phi_{rz}&=2z\phi_{zz}(f_6-f_3\frac{z^2}{2}) \label{eq:B002}
\\
   z\phi_{x^0z} - \phi_{x^0}  &=2\phi_{zz}(f_7\frac{z^2}{2}+f_8),\label{eq:B003}\\
   z\phi_{x^0} &= 2\left[zf_{10}\phi +  \phi_z (f_7\frac{z^2}{2} + f_ 8)\right]\label{eq:B004}\\
   \frac{1}{r}\phi_r&=2\left[(f_9 + f_4 + f_3\frac{z^2}{2})\phi + z\phi_z(f_6-f_3\frac{z^2}{2})\right]\label{eq:B005}
\end{align}
Deriving \eqref{eq:B001} in relation to $z$ and adding with \eqref{eq:B002} we have
\begin{align}\label{eq:BcaseI}
    \phi_{zz}(f_6-f_4-f_3z^2) + (\phi-z\phi_z)f_3=0
\end{align}
which brands in two cases:

{\bf Case I:}  If $f_6-f_4-f_3z^2\neq 0$ (i.e. $f_3\neq 0$ and  $ f_4\neq f_6$).
 From \eqref{eq:BcaseI}, 
 \begin{align*}
     \phi &= -z\int \frac{K}{z^2\sqrt{|f_6-f_4-f_3z^2|}}dz\\
     &=\frac{K}{f_6-f_4}\sqrt{|f_6-f_4 - f_3z^2|} - f_{11}z
 \end{align*}
 which means that $F=u\phi$ is Randers type.

 {\bf Case II:} If $f_6-f_4-f_3z^2=0$ (i.e. $f_3= 0$ and  $ f_4= f_6$).
 Adding equation \eqref{eq:B001} with \eqref{eq:B005}, we conclude  that $f_9=-2f_4$. Adding \eqref{eq:B003} with the derivative of \eqref{eq:B004} in relation to $z$ and using \eqref{eq:B004} we have,
 \begin{align}\label{eq:BCaseII}
     zf_{10}\phi - \phi_z(f_{10}z^2-2f_8) = 0.
 \end{align}
 From equation \eqref{eq:BCaseII} we have two subcases,
 
 {\bf Case 2.1:} If $f_{10}z^2-2f_8\neq 0$, then solving \eqref{eq:BCaseII}, we conclude that $\phi $ it is of Riemannian type.

 {\bf Case 2.2:} If $f_8=f_{10}=0.$ Then, the system \eqref{eq:B001}-\eqref{eq:B005} reduces to
\begin{align*}
    -\frac{\phi_r}{r} &= 2(\phi-z\phi_z)f_4\\
    \phi_{x^0}&=z\phi_z f_7.
\end{align*}
\end{proof}

\begin{corollary}\normalfont Let $F= u \phi(x^0,r,z),$ be a Finsler metric defined on $I\times \mathbb{B}^n(\rho)$, $n\geq 3$, where $z=\frac{y^0}{u},$ $r=\vert \overline{x}\vert$, $s=\frac{\langle \overline{x},\overline{y}\rangle}{u}$ and $u=\vert \overline{y}\vert$ and $TM$ defined with coordinates \eqref{coordx}, \eqref{coordy}. Then $F$ has vanishing Landsberg curvature if, and only if, the positive function $\phi$ satisfies
\begin{align}
    \phi_z N_{z z z}+s \Omega U_{z z z}+\Omega W_{z z z}=0,\label{eq:Lands1}\\
    \phi_z N_{s z z}+\Omega W_{s z z}=0,\label{eq:Lands2}\\
    z \phi_z \Psi\left(\frac{N_z}{z}\right)+z s \Omega \Psi\left(\frac{U_z}{z}\right)+\Omega \Psi\left(W_z\right)=0, \label{eq:Lands3}\\
    z\phi_z\Psi\left(\frac{N_s}{z}\right) + \Omega\Psi\left(W_s\right)=0. \label{eq:Lands4}
\end{align}
\end{corollary}
\begin{proof}
Let the orthonormal matrix $O\in O(n)$ considered in  the proof of Theorem \ref{Thm2:B0}. From definitions of $U, N$ and $W$ in \eqref{def:U}, \eqref{def:N} and \eqref{def:W}, we have, $U_s=W_{ss}=N_{ss}=0$. Then, from $L_{000}=0$ we obtain \eqref{eq:Lands1}. From $L_{002}=0, $ we conclude in \eqref{eq:Lands2}. From $L_{033}=0, $ we obtain \eqref{eq:Lands3} and finally, from $L_{332}=L_{331}=0$ we obtain \eqref{eq:Lands4}. 

Conversely, suppose that equations \eqref{eq:Lands1}-\eqref{eq:Lands4} are satisfied; then, from \eqref{eq:Lands1}, \eqref{eq:Lands2} and \eqref{eq:Lands3} we have 
\begin{align}
    \phi_z\Psi(N_{zz}) + s\Omega\Psi(U_{zz}) + \Omega \Psi(W_{zz})&=0 \label{eq:I}
\end{align}
then, $L_{00l}=L_{000}=0$. 
From \eqref{eq:Lands3}, 
\begin{align*}
    z\phi_z\Psi\left(\frac{N_z}{z}\right) + sz\Omega \Psi\left(\frac{U_z}{z}\right) + \Omega \Psi(W_z) = 0.
\end{align*}

From \eqref{eq:Lands2} \begin{align}\label{eq:II}\phi_z\Psi(N_{sz}) + \Omega \Psi(W_{sz}) = 0\end{align}
Using \eqref{eq:betaa1}, \eqref{eq:betaa2}, \eqref{eq:I}, \eqref{eq:II} and \eqref{eq:Lands3} we have
\begin{align}
    \frac{1}{z}\phi_z\Psi\left(z^2\Psi\left(\frac{N_z}{z}\right)\right) + \frac{s}{z}\Omega\Psi\left(z^2\Psi\left(\frac{U_z}{z}\right)\right) + \frac{\Omega}{z^2}\Psi\left(z^2\Psi(W_z)\right) + 2\Omega\Psi(W_z)=\quad\label{eq:V}\\
    =-s(\phi_z\Psi(N_{sz} + \Omega \Psi(W_{sz}))) - \left(z\phi_z\Psi\left(\frac{N_z}{z}\right) + sz\Omega \Psi\left(\frac{U_z}{z}\right) + \Omega\Psi(W_z)\right)\nonumber\\
    -z\left(\phi_z\Psi(N_{zz}) + s\Omega\Psi(U_{zz})+ \Omega \Psi(W_{zz})\right) = 0\nonumber
\end{align}
Then, $L_{0kl}=0.$

From \eqref{eq:Lands4} we obtain, \begin{align}\label{eq:green}
z\phi_z\Psi\left(\frac{N_s}{z}\right) + \Omega \Psi(W_s) = 0.
\end{align}

Using \eqref{psieq1}, \eqref{eq:betaa1} and  \eqref{eq:Lands3}, \eqref{eq:Lands4}, \eqref{eq:green} we have

\begin{align}
\phi_z\Psi\left(z^2\Psi\left(\frac{N}{z^2}\right)\right) + s\Omega \Psi\left(z^2\Psi\left(\frac{U}{z^2}\right)\right) + \frac{\Omega}{z}\Psi\left(z^2\Psi\left(\frac{W}{z}\right)\right) + z\Omega \Psi\left(\frac{W}{z}\right)=\quad \label{eq:IV}\\
-s\left[z\phi_z\Psi\left(\frac{N_s}{z}\right) + \Omega \Psi(W_s)\right] - z\left[z\phi_z\Psi\left(\frac{N_z}{z}\right) + sz\Omega \Psi\left(\frac{U_z}{z}\right) + \Omega\Psi(W_z)\right] = 0\nonumber
\end{align}

From \eqref{eq:Lands4} and \eqref{eq:II},

\begin{align}
    \frac{\phi_z}{z}\Psi\left(z^2\Psi\left(\frac{N_s}{z}\right)\right) + \frac{\Omega}{z^2}\Psi\left(z^2\Psi(W_s)\right) + 2\Omega\Psi(W_s) =\qquad \label{eq:III}\\
    -z\phi_z\Psi\left(\frac{N_s}{z}\right) - \Omega\Psi(W_s) - z\left[\phi_z\Psi(N_{sz}) + \Omega\Psi(W_{sz})\right]=0.\nonumber
\end{align}

Finally, using \eqref{eq:psi3I}, \eqref{eq:psi3II}, \eqref{eq:IV}, \eqref{eq:III}, \eqref{eq:V}, and \eqref{eq:betaa1}, we have
\begin{align*}
\frac{\phi_z}{z^2}\Psi\left(z^2\Psi\left(z^2\Psi\left(\frac{N}{z^2}\right)\right)\right) + \frac{s\Omega}{z^2}\Psi\left(z^2\Psi\left(z^2\Psi\left(\frac{U}{z^2}\right)\right)\right) +\quad\\
+\frac{\Omega}{z^2}\Psi\left(z^2\Psi\left(z^2\Psi\left(\frac{W}{z}\right)\right)\right) + \frac{3\Omega}{z}\Psi\left(z^2\Psi\left(\frac{W}{z}\right)\right)=\\
-\left[2\phi_z\Psi\left(z^2\Psi\left(\frac{N}{z^2}\right)\right) + 2s\Omega \Psi\left(z^2\Psi\left(\frac{U}{z^2}\right)\right)\right]\quad \\ 
- \left[\frac{s\phi_z}{z}\Psi\left(z^2\Psi\left(\frac{N_s}{z}\right)\right) + \frac{s\Omega}{z^2}\Psi\left(z^2\Psi\left(W_s\right)\right)\right]\quad  \\
-\left[\phi_z\Psi\left(z^2\Psi\left(\frac{N_z}{z}\right)\right) + s\Omega \Psi\left(z^2\Psi\left(\frac{U_z}{z}\right)\right) + \frac{\Omega}{z}\Psi\left(z^2\Psi\left(W_z\right)\right)\right]=
\\
2\left[\frac{\Omega}{z}\Psi\left(z^2\Psi\left(\frac{W}{z}\right) - z\Omega\Psi\left(\frac{W}{z}\right)\right)\right]+ 2s\Omega \Psi\left(W_s\right) + 2z\Omega\Psi\left(W_z\right)  = 0,
\end{align*}
which prove $L_{jkl}=0.$

\textbf{Proof of  Theorem \ref{theo:main}:}
\end{proof}
Along the proof, $f_i:I\times \mathbb{R}\to \mathbb{R}$ are differentiable functions of $x^0$ and $r.$
 Deriving \eqref{eq:Lands3} in relation to $s$, we have
 $z(\phi_zN_{szz} + \Omega W_{szz}) + \Omega W_{sz} = 0$ and using \eqref{eq:Lands2} and \eqref{eq:Lands4} we conclude \begin{align}\label{eq:Wsz}
 W_{sz}&=0,& N_{szz}&=0, & z\Psi\left(\frac{N_s}{z}\right)=N_s-zN_{sz}&=0 
 \end{align} 
Deriving \eqref{eq:Lands1} in relation to $s$, we obtain $U_{zzz}=0$ and then there are $f_{i}=f_{i}(x^0, r), \, \, i=1,2,3$ such that 
\begin{align}\label{eq:U}
    U=f_{1}+f_{2} z+f_{3}\frac{z^2}{2}.
\end{align}

From  $N_{ss}=0$, $W_{ss}=0$ and \eqref{eq:Wsz}  there are $N_{0}=N_{0}(x^0, r,z)$, $W_{0}=W_{0}(x^0,r,z)$, $f_{4}=f_{4}(x^0, r)$, $f_{5}=f_{5}(x^0, r)$ such that 
\begin{align}
   N&=N_{0}+f_{4}sz. \label{eq:N_1} \\
   W&=W_0(x^0,r,z) + sf_5(x^0,r) \label{eq:N_2}
\end{align}

From \eqref{eq:Lands1} and \eqref{eq:Lands4} we obtain conditions on $N_{0}$ and $W_{0}$:
\begin{align}
    \phi_zN_{0zzz} + \Omega W_{0zzz} &= 0\\
    \phi_z(N_{0z} - zN_{0zz})-\Omega (z W_{0zz}) + \Omega f_3(x^0,r)&=0.
\end{align}

As in the proof of Theorem \ref{Thm2:B0} and Corollary  \ref{cor2}, we have that $\phi>0$ must satisfy

\begin{align}
    -\frac{1}{r}\phi_r &= 2(\phi-z\phi_z)(f_1+f_2z+f_3z^2/2)\label{eq:first}\\
    z\phi_{x^0z} - \phi_{x^0}  &=2\phi_{zz}(N_0-zW_0)\label{eq:second}\\
    \frac{\phi_{rz}}{r}&=2z\phi_{zz}(f_4-f_5-f_1-f_2z-f_3z^2/2)\label{eq:thirt}\\
    z\phi_{x^0}&=2[W_0\phi + \phi_z(N_0-zW_0)]\label{eq:forth}\\
    \frac{\phi_r}{r}&=2(\phi-z\phi_z)(f_5+f_1+f_2z+f_3z^2/2)+2z\phi_zf_4 \label{eq:fifth}
    \end{align}
    with the additional conditions,
    \begin{align}
        \phi_zN_{0zzz} + \Omega W_{0zzz} &= 0 \label{eq:additonal1}\\
    \phi_z(N_{0z} - zN_{0zz})-\Omega (z W_{0zz}) + \Omega f_3&=0.\label{eq:additonal2}
\end{align}

Adding and subtracting \eqref{eq:first} and \eqref{eq:fifth}
\begin{align}
    0&=(\phi-z\phi_z)(2f_1+f_5 + 2f_2z + 2f_3\frac{z^2}{2}) + z\phi_zf_4\\
    \frac{1}{r}\phi_r&=(\phi-z\phi_z)f_5 + z\phi_z f_4
\end{align}
Subtracting them,
\begin{align}
    \frac{1}{r}\phi_r=-2(\phi-z\phi_z)(f_1+f_2z + f_3\frac{z^2}{2})
\end{align} 
 deriving in relation to $z$ and comparing with \eqref{eq:thirt}, we have

\begin{align}\label{eq:F}
-z\,\phi_{zz}
\Bigl(
f_4 - f_5 - 2f_1 - 2f_2 z - f_3 z^2
\Bigr)
=
\Omega\,(f_2 + f_3 z).
\end{align}

{\bf Case 1:} $f_4 - f_5 - 2f_1 - 2f_2 z - f_3 z^2\neq 0$

\begin{align}
    \frac{\Omega_z}{\Omega} = \frac{f_3+f_4z}{f_1 - f_5 - 2f_2 - 2f_3 z - f_4 z^2}
\end{align}
integrating in $z$ and using the identity $\Omega = -z^2\left(\frac{\phi}{z}\right)_z$ we conclude that $F$ is Randers form. 

{\bf Case 2:} $ f_4 - f_5 - 2f_1 - 2f_2 z - f_3 z^2 =  0$ ($f_2=f_3=0$ and $f_4 - f_5 - 2f_1=0$)
From \eqref{eq:first} and \eqref{eq:fifth} we obtain
\[
f_4 = 2f_1 + f_5 = 0.
\]
Using the conditions \eqref{eq:additonal1}--\eqref{eq:additonal2}, together with \( f_3 = 0 \) and assuming
\( N_{0zzz} \neq 0 \) (which implies \( N_{0z} - z N_{0zz} \neq 0 \)), we deduce that
\begin{align}
\frac{\phi_z}{\Omega}
= -\,\frac{W_{0zzz}}{N_{0zzz}}
= \frac{z W_{0zz}}{N_{0z} - z N_{0zz}} .
\end{align}
Consequently, there exists a function \( f_6 = f_6(x^0,r) \) such that
\begin{align}\label{eq:W0zzf6}
W_{0zz} = \pm f_6 \bigl( N_{0z} - z N_{0zz} \bigr).
\end{align}
Substituting \eqref{eq:W0zzf6} into \eqref{eq:additonal1} (or equivalently into
\eqref{eq:additonal2}), we obtain
\begin{align*}
\phi_z \mp z \Omega f_6 &= 0, \\
\pm z \phi f_6 - \phi_z \bigl( 1 \pm z^2 f_6 \bigr) &= 0.
\end{align*}
Then, there exist $f_7=f_7(x^0,r)$ such that
\begin{align}
    \phi = f_7 \sqrt{|1\pm f_6 z^2|}
\end{align}
 which means that $F=u\phi$ is Riemannian.

 On the other hand, if $N_{0z} - z N_{0zz}=0$, then,
 \begin{align}
     N_{0}=f_8 + f_9\frac{z^2}{2}\\
     W_{0}=f_{10} + f_{11}z
 \end{align}
Deriving \eqref{eq:forth} in relation to $z$, and subtracting \eqref{eq:second}, we have, $ \phi_{x^0} = f_{11}\phi +  z(f_9 -f_{11})\phi_z $. Thus, from \eqref{eq:forth},
\begin{align}\label{eq:landsnonberw}
    [f_{11}z + 2f_{10}]\phi =[f_{11}z^2+2f_{10}z - 2f_8]\phi_z.
\end{align}
From this, we consider two sub-cases:

{\bf Case 2.1:} If $f_{11}z^2+2f_{10}z - 2f_8=0$, then $f_{11}=0, \, f_{10}=0, \, f_8=0$, which reduces the Landsberg system given in Corollary 3 into,
\begin{align*}
    -\frac{1}{r}\phi_r &= 2(\phi-z\phi_z)f_1\\
\phi_{x^0}&= z\phi_z f_9
\end{align*}
which means that $F=u\phi$ is Berwald metric by Theorem \ref{cor2}.

{\bf Case 2.2:} If $f_{11}z^2+2f_{10}z - 2f_8 \neq 0$.

From \eqref{eq:landsnonberw} we obtain
\begin{align}\label{eq:lastcase}
\frac{\phi_z}{\phi}
=
\frac{f_{11} z + 2 f_{10}}{f_{11} z^2 + 2 f_{10} z - 2 f_8}.
\end{align}

Considering $\Delta=-f_{10}^2-2 f_{11} f_8>0,$
and integrating \eqref{eq:lastcase} with respect to $z$, we  obtain
\begin{align}
    \phi = \sqrt{| f_{11}z^2+2f_{10}z - 2f_8|} \times e^{f_{12} + \int_0^z \frac{f_{10}}{f_{11}z^2+2f_{10}z - 2f_8}dz}
\end{align}
or
\begin{align}\label{eq:philastcase}
\displaystyle\phi
&=
k(x^0,r)\;
\sqrt{\zeta^2+\alpha^2}\;
\times e^{
\beta\,\arctan\!\left(\frac{\zeta}{\alpha}
\right)},
\end{align}
where,
\begin{align}
\zeta&=z+\frac{f_{10}}{f_{11}}=z+\alpha\beta&
\alpha &= \frac{\sqrt{\Delta}}{f_{11}},& 
\beta &= \frac{f_{10}}{\sqrt{\Delta}}.
\end{align}
\qed

\begin{remark}\normalfont
The function $F=u\phi$, where $\phi$ is given by \eqref{eq:philastcase}, is not regular on $TM\setminus\{0\}$.

Indeed, assume by contradiction that $F$ is regular on $TM\setminus\{0\}$. Fix $y^0=1$ and consider $\overline{y}=(t,0,\ldots,0)$ with $t\in\mathbb{R}\setminus\{0\}$. Then the function
\[
\theta(t)=\vert t\vert\phi(x^0,r,|t|^{-1})
\]
should be smooth at $t=0$.

For simplicity, denote
\[
k=k(x^0,r),\qquad
M(t)=\alpha^2(1+\beta^2)t^2+2\alpha\beta|t|+1,\qquad
N(t)=\beta\arctan\!\left(\frac{1}{\alpha|t|}+\beta\right).
\]
Then
\begin{equation}\label{Fnreg1}
    \theta(t)=k\sqrt{M(t)}\,e^{N(t)}.
\end{equation}

Note that
\[
\lim_{t\to0}\theta(t)=k\,e^{\beta\pi/2}.
\]

Differentiating \eqref{Fnreg1}, we obtain
\begin{equation}\label{Fnreg2}
    \theta'(t)
    =
    \alpha^2(1+\beta^2)\frac{t\,\theta(t)}{M(t)}.
\end{equation}

Differentiating \eqref{Fnreg2}, we get
\begin{equation}\label{Fnreg3}
    \theta''(t)
    =
    \alpha^2(1+\beta^2)\frac{\theta(t)}{M(t)^2}.
\end{equation}

Differentiating once more, we obtain
\begin{equation}\label{Fnreg4}
\theta'''(t)
=
\alpha^2(1+\beta^2)\frac{\theta(t)}{M(t)^3}
\left[
-3\alpha^2(1+\beta^2)t
-4\alpha\beta\frac{d}{dt}(|t|)
\right].
\end{equation}

Finally, observe that $\theta'''(t)$ does not admit a limit as $t\to0$, since
\[
\lim_{t\to0^+}\theta'''(t)
=
\alpha^2(1+\beta^2)\,k\,e^{\beta\pi/2}\,(-4\alpha\beta),
\]
while
\[
\lim_{t\to0^-}\theta'''(t)
=
\alpha^2(1+\beta^2)\,k\,e^{\beta\pi/2}\,(4\alpha\beta).
\]

Therefore, $\theta$ is of class $C^2$ but not $C^3$ at $t=0$, which contradicts the regularity of $F$. Hence, $F$ is not regular on $TM\setminus\{0\}$.
\end{remark}
\section{Non Belwardian pseudo Finsler metric with vanishing Landsberg curvature}

From \eqref{eq:first} ($  -\frac{1}{r}\phi_r = 2(\phi-z\phi_z)f_1$) and \eqref{eq:lastcase} we obtain
\begin{align}\label{eq:phir/r}
    \frac{1}{r}\frac{\phi_r}{\phi} = \frac{4f_1f_8}{f_{11} z^2 + 2 f_{10} z - 2 f_8}.
\end{align}
Using \eqref{eq:philastcase},
\begin{align}\label{eq:phirsobr}
    \frac{1}{r}\left(Q\frac{k_r}{k} + f_{11}(\zeta \zeta_r + \alpha\alpha_r + \beta(\alpha \zeta_r-\zeta \alpha_r)) + Q\beta_r \arctan\left(\frac{\zeta}{\alpha}\right)\right) = 4f_1f_8,
\end{align}
where, $Q= f_{11}z^2+2f_{10}z - 2f_8$. Deriving \eqref{eq:phirsobr} in relation to $z$, we obtain
\begin{equation}\label{eqz:phirsobr}
    (2f_{11}z+2f_{10})\left[\dfrac{k_r}{k}+\beta_r\arctan\left(\frac{\zeta}{\alpha}\right)\right]+f_{11}[\zeta_r+\alpha\beta_r-\beta\alpha_r]=0.
\end{equation}
Deriving \eqref{eqz:phirsobr} in relation to $z$ and since $f_{11}\neq0$, we obtain
\begin{equation}\label{eqzz:phirsobr}
    \dfrac{k_r}{k}+\beta_r\left[(f_{11}z+f_{10})\frac{\alpha}{Q}+\arctan\left(\frac{\zeta}{\alpha}\right)\right]=0.
\end{equation}
Deriving \eqref{eqzz:phirsobr} in relation to $z$, we obtain
\[\beta_r\left[f_{11}\frac{\alpha}{Q}-(f_{11}z+f_{10})^2\frac{\alpha}{Q^2}\right]=0,\]
or equivalently 
\begin{equation}\label{eqzzz:phirsobr}
\beta_r\alpha\Delta=0.
\end{equation}
Since $\alpha,\Delta\neq0$, by \eqref{eqzzz:phirsobr}, we have \begin{equation}\label{eq:betar0}
    \beta_r=0.
\end{equation} Now, replacing \eqref{eq:betar0} in \eqref{eqzz:phirsobr}, we obtain $k_r=0$. Consequently, by \eqref{eq:phirsobr}, we have 
\begin{equation}\label{eq:f10f8f11}
    \frac{f_{11}}{r}\left[\zeta \zeta_r + \alpha\alpha_r + \beta(\alpha \zeta_r-\zeta \alpha_r)\right]= 4f_1f_8.
\end{equation}
On the other hand, by \eqref{eq:betar0}, we have
\[\left(\dfrac{f^2_{10}}{f_8f_{11}}\right)_r=0,\]equivalently,
\[\dfrac{f^2_{10}}{f_8f_{11}}=2\eta(x^0).\]
With this, $\Delta = -2f_8f_{11}(1+\eta(x^0)), \quad \beta^2 =-\frac{2\eta(x^0)}{1+\eta(x^0)}, \quad \alpha^2 = -2\frac{f_8}{f_{11}}(1+\eta(x^0)) $. Then, using $\zeta - \alpha\beta= z$, equation \eqref{eq:f10f8f11} is reduced to,
\begin{align*}
    \frac{f_{11}}{r}\alpha \alpha_r(1+\beta^2) &= 4f_1f_8,\\
    -\frac{1}{r}\frac{f_{11}}{f_8}\left(\frac{f_8}{f_{11}}\right)_r(1-\eta(x^0))&=4f_1,\\
    f_8&=-f_{11}e^{-\frac{\int4rf_1\,dr}{1-\eta(x^0)}}.
\end{align*}

Thus, if $F=u\phi$ is Landsberg type, then
\begin{align}\label{eq:phinonBe}
\displaystyle\phi
&=
k(x^0)\;
\sqrt{\zeta^2+\alpha^2}\;
\times e^{
\beta\,\arctan\!\left(\frac{\zeta}{\alpha}
\right)},
\end{align}
with, 
\begin{align} 
\beta_{r} &= \left[\frac{f_{10}}{\sqrt{\Delta}}\right]_r=0.
\end{align}

On the other hand, if $F=u\phi$, where $\phi$ is given by \eqref{eq:phinonBe}, then it is Berwald type if $\phi$  satisfies $\frac{\phi_{x^0}}{\phi} = z\frac{\phi_z}{\phi} f_9$ for some $f_9=f_9(x^0,r)$. Then, from \eqref{eq:lastcase}, \eqref{eq:phinonBe}, and $\zeta \beta +\alpha = z\beta - \frac{f_8}{\sqrt{\Delta}}$,
\begin{align}\label{eq:LnonB}
    \frac{k'}{k} + \frac{\alpha}{\alpha^2+ \zeta^2}\left[z\beta  - \frac{f_8}{\sqrt{\Delta}}\right]_{x^0} + \beta_{x^0}\arctan\left(\frac{\zeta}{\alpha}\right) = zf_9\frac{\zeta + \alpha\beta}{\zeta^2+\alpha^2},
\end{align}
 multiplying equation \eqref{eq:LnonB} both of sides by $\zeta^2+\alpha^2$, deriving twice in $z,$ and using $\left[\zeta-\alpha \beta\right]_{x^0} = 0$,

 \begin{align*}
     \frac{k'}{k} + \beta_{x^0}\arctan\left(\frac{\zeta}{\alpha}\right) + \frac{\zeta\alpha\beta_{x^0}}{\zeta^2 + \alpha^2}=f_9.
 \end{align*}
 Deriving in relation to $z$, we conclude that $\beta_{x^0}=0,$ and then $f_9=\frac{k'}{k}.$ With this, from \eqref{eq:LnonB}, we have the condition on $f_8, f_9$ and $f_{10}$: 
 \begin{align*}
     \left[\frac{f_8}{f_{10}}\right]_{x^0} + 2f_9\left[\frac{f_8}{f_{10}}\right]=0.
 \end{align*}

\end{document}